\documentclass[twoside,a4paper,reqno,11pt]{amsart} 
\usepackage[top=30mm,right=30mm,bottom=30mm,left=30mm]{geometry}
\usepackage[utf8]{inputenc}
\usepackage{amsthm}
\usepackage{amsfonts}
\usepackage{amsmath}
\usepackage{ amssymb }
\usepackage{todonotes}
\usepackage{hyperref} 
\usepackage{enumerate}%
\usepackage{comment}
\usepackage[foot]{amsaddr}
\newtheorem{theorem}{Theorem}
\newtheorem{proposition}[theorem]{Proposition}
\newtheorem{corollary}[theorem]{Corollary}
\newtheorem*{theorem*}{Theorem}
\newtheorem{lemma}[theorem]{Lemma}
\theoremstyle{definition}
\newtheorem{definition}[theorem]{Definition}
\newtheorem*{definition*}{Definition}
\newtheorem*{proposition*}{Proposition}
\newtheorem{question}[theorem]{Question}
\theoremstyle{remark}
\newtheorem{remark}[theorem]{Remark}
\newtheorem{example}[theorem]{Example}
\numberwithin{equation}{section}

\theoremstyle{plain}
\newtheorem{manualtheoreminner}{Theorem}
\newenvironment{manualtheorem}[1]{%
  \renewcommand\themanualtheoreminner{#1}%
  \manualtheoreminner
}{\endmanualtheoreminner}


\title{Quasi-powerful $p$-groups}
\author{James Williams}
\address{J. Williams, School of Mathematics, University of Bristol, Bristol BS8 1UG, UK}
\email{j.l.i.williams@bristol.ac.uk}
\date{\today}

\begin{document}

\maketitle

\begin{abstract}
 In this paper we introduce the notion of a quasi-powerful $p$-group for odd primes $p$. These are the finite $p$-groups $G$ such that $G/Z(G)$ is powerful in the sense of Lubotzky and Mann. We show that this large family of groups shares many of the same properties as powerful $p$-groups. For example, we show that they have a regular power structure, and we generalise a result of Fern\'andez-Alcober on the order of commutators in powerful $p$-groups to this larger family of groups. We also obtain a bound on the number of generators of a subgroup of a quasi-powerful $p$-group, expressed in terms of the number of generators of the group. We give an infinite family of examples which demonstrates this bound is close to best possible.
\end{abstract}

\section{Introduction}

 It can be said that the modern study of finite $p$-groups began with the groundbreaking paper of P. Hall, published in  1933 \cite{phallcontribtothetheoryofgroupsofprimepowerorder}. In this paper Hall introduced the notion of a \textit{regular} $p$-group (see Definition \ref{defn of regular p group}) and he showed that these groups have desirable properties and a theory which in some sense parallels that of abelian groups. The study of families of finite $p$-groups with certain desirable properties continues to this day.

The \textit{powerful} $p$-groups, introduced by Lubotzky and Mann in \cite{LUBOTZKY1987484}, are another well-studied family of groups with abelian-like properties. Powerful $p$-groups also have many abelian-like properties. On the one hand they can be thought of as very similar to abelian groups, but on the other hand they can be thought of as close to a typical $p$-group (see Remark \ref{abelian likeness of powerful $p$-groups}). Given this, it is not surprising that the theory of powerful $p$-groups has found many applications, as problems concerning typical $p$-groups can often be reduced to questions about powerful $p$-groups. Perhaps the most widely celebrated application  is Shalev's proof of the coclass conjectures \cite{Shalev1994}. However the impact of powerful $p$-groups  is extensive and even stretches beyond finite $p$-groups (for more applications, see Remark \ref{end notes further references}). \par 
 
 Given the widespread usefulness  of powerful $p$-groups, it is natural to seek a larger family of groups with similar properties.  With this goal in mind, we introduce \textit{quasi-powerful} $p$-groups, and we extend many of the remarkable properties of powerful $p$-groups to this family.
 \begin{definition*}
 Let $p$ be an odd prime. We say that a finite $p$-group $G$ is  \textit{quasi-powerful} if $G/Z(G)$ is a powerful $p$-group.
 \end{definition*}
 This family is larger than that of powerful $p$-groups. For instance it contains all powerful $p$-groups and also all $p$-groups of nilpotency class $2$, and therefore unifies two large families of groups which are known to have abelian-like properties. Informally we can think of groups in this family as being close to powerful $p$-groups; in a powerful group any commutator is equal to a $p$th power in the group, whereas in a quasi-powerful group any commutator is equal to the product of a $p$th power and an element in the centre. It seems reasonable to expect that many properties of powerful $p$-groups can be adapted to this family of groups and in this paper we show this to be the case. It is our hope that quasi-powerful $p$-groups will be useful as a tool in inductive arguments and reductions. 
 
 In order to state our main results, we need to recall some notation and terminology. 
 Let $G$ be a finite $p$-group.  The $i$th Omega subgroup of $G$ is defined by
 $$ \Omega_{i}(G)=\langle g \in G \mid g^{p^{i}}=1\rangle. $$
 Notice that this coincides with the set of elements whose order divides $p^{i}$ if and only if for any two elements  $a,b \in G$ each of order at most $p^{i}$ we have that the order of the product $ab$ is at most $p^{i}$. Another equivalent formulation of this is that $\exp \Omega_{i}(G) \leq p^{i}$. This is clearly true for abelian $p$-groups, however this is not true for $p$-groups in general. For example, $\Omega_{1}(D_{8})$ contains an element of order $4$.
 
 The $i$th Agemo subgroup of $G$ is defined by $$ G^{p^{i}}=\langle g^{p^{i}} \mid g \in G \rangle. $$
 Sometimes this subgroup is denoted as $\mho_{i}(G)$. Note that the set of $p^{i}$th powers need not coincide with the group it generates. However, in the abelian setting they do coincide. 
 
 As the terminology and notation indicates, the Omega and Agemo subgroups are in some sense dual to each other. Indeed, for an abelian $p$-group we have that $|G:G^{p^{i}}|=|\Omega_{i}(G)|$ for all integers $i \geq 0$. 
 
These ideas motivate the notion of a \textit{regular power structure}. 
 
\begin{definition*} 
A finite $p$-group $G$ has a \textit{regular power structure} if the following three conditions hold for all positive integers $i$:
\begin{align}
 G^{p^{i}} &=\{g^{p^{i}} \mid g \in G \}. \label{eqn pth power} \\
      \Omega_{i}(G) &= \{g\in G \mid o(g) \leq p^{i} \}. \label{eqn omega conditon} \\
 |G:G^{p^{i}}| &= | \Omega_{i}(G)| \label{eqn index condition}.
\end{align}
\end{definition*}
\vspace{2mm}
We have seen above that abelian $p$-groups have a regular power structure. P. Hall showed in \cite{phallcontribtothetheoryofgroupsofprimepowerorder} that regular $p$-groups have a regular power structure. 
For odd primes $p$, powerful $p$-groups  also have a regular power structure - the first property (\ref{eqn pth power}) was established in \cite{LUBOTZKY1987484}, with the latter two properties (\ref{eqn omega conditon}) and (\ref{eqn index condition}) proved by Wilson in \cite{L2002} (using a result of  H\'{e}thelyi and  L\'{e}vai \cite{HETHELYI20031}). An alternate and shorter proof of these last two facts is given by  Fern\'andez-Alcober in  \cite{Fernandez-Alcober2007}. Another independent proof that (\ref{eqn index condition}) holds in powerful $p$-groups is given by  Mazur in \cite{Mazur2007}.

In \cite{GONZALEZSANCHEZ2004193}  Gonz\'{a}lez-S\'{a}nchez and  Jaikin-Zapirain introduce  the family of  \textit{potent} $p$-groups (we recall this  notion in Definition \ref{defn potent pgroup}), and they show that this family of groups  has a regular power structure (this property is called  \textit{power abelian} in \cite{GONZALEZSANCHEZ2004193}). Hence we see there is a sustained and significant interest not only in finding families of finite $p$-groups with a regular power structure, but also in exploring different approaches to the problem. 

Our first main result reveals that quasi-powerful $p$-groups have a regular power structure. 
\begin{theorem} \label{theorem intro regular pow struc and 3 points}
Let $p$ be an odd prime, let $G$ be a quasi-powerful $p$-group and let $i$ be a non-negative integer.
\begin{enumerate}[(i)]
    \item If $a,b \in G$ both have order at most $p^{i}$ then their product $ab$ has order at most $p^{i}$. \label{thm intro: product of two elements of order at most p^i has order at most p^i}
    \item The group generated by the $p^{i}$th powers coincides with the set of $p^{i}$th powers. \label{theorem letter group generated by pith powers coincides with set of powers}
    \item  $|G:G^{p^{i}}| = | \Omega_{i}(G)|$.  \label{theorem letter regular power structure last condition on order of G}
\end{enumerate}
In particular, $G$ has a regular power structure.
\end{theorem}

We will show that if $p>3$, then every quasi-powerful $p$-group is potent; thus for $p>3$ the fact they have a regular power structure follows from \cite{GONZALEZSANCHEZ2004193}. However all results given in this paper will be proved independently of this fact and moreover for us the difficult and most interesting case is $p=3$ and the quasi-powerful $3$-groups need not be potent (see Example \ref{example of a quasipowerful group}). 

It is well known that if $G$ is a powerful $p$-group then $G^{p}$ is also powerful. In fact, this is true for all known families of groups satisfying property (\ref{eqn pth power}). It is an open question if the same condition holds for any finite $p$-group satisfying (\ref{eqn pth power}). This question was posed by Wilson in \cite{L2002} and he noted that an affirmative answer to this question would  answer  a question of Shalev \cite[Problem 13]{Shalev1995} (see Question \ref{question of shalev}). Since quasi-powerful $p$-groups satisfy (\ref{eqn pth power}),  Wilson's question provides additional motivation for our next result. 

\begin{theorem}
\label{theorem letter group generated by pith powers is powerful pgroup}
Let $G$ be a quasi-powerful $p$-group. Then $G^{p^{i}}$ is a powerful $p$-group for all $i \geq 1$.
\end{theorem}

To understand finite $p$-groups, one needs to study the interactions between powers and commutators. For example, the defining properties of all of the families of groups described above all involve  some condition on commutators in terms of groups of $p$th powers. Theorem 1 in \cite{Fernandez-Alcober2007}  gives bounds on the order of a commutator in a powerful $p$-group in terms of the components within the commutator (see Theorem \ref{ Gustavos result for powerful p-groups}). These results turn out to be very useful when working with $p$th powers in powerful $p$-groups, and played a key role in \cite{Williams2019} and \cite{Williams2018} to show that certain normal subgroups of powerful $p$-groups are \textit{powerfully nilpotent} (see Section \ref{subsection powerfully nilpotent groups}).

We generalise these  bounds on commutators to quasi-powerful $p$-groups. The first result shows that in a quasi-powerful $p$-group the order of a commutator is bounded by the order of its components. 

\begin{theorem}
\label{Theorem letter. Quasi powerful order comm less than order elems}
Let $G$ be a quasi-powerful $p$-group and $x,y \in G$. Then $o([x,y]) \leq o(y)$.
\end{theorem}

The next result allows us to say more if we know how elements can be expressed as $p$th powers. 

\begin{theorem}
\label{ theorem letter more detailed commutator order bound}
Let $G$ be a quasi-powerful $p$-group. If $x,y \in G$ are such that $o(x) \leq p^{i+1}$ and $o(y) \leq p^{i}$, then $o([x^{p^{j}}, y^{p^{k}}]) \leq p^{i-j-k}$ for all $j, k \geq 0.$
\end{theorem}
An interesting point to note is that we apply Theorem \ref{theorem intro regular pow struc and 3 points}(\ref{thm intro: product of two elements of order at most p^i has order at most p^i}) to prove Theorem \ref{ theorem letter more detailed commutator order bound}. The argument we use is quite general and can be applied to prove an analagous result for potent $p$-groups for odd primes $p$, hence we can obtain the following theorem. \par

\begin{theorem}
\label{theorem intro detailed comm bound for potent groups}
Let $p$ be an odd prime and $G$ be a potent $p$-group. Suppose $x,y \in G$ with $o(x) \leq p^{i+1}$ and $o(y) \leq p^{i}$, then $o([x^{p^{j}},y^{p^{k}}]) \leq p^{i-j-k}$ for all $j, k \geq 0.$.
\end{theorem}

One of the most abelian-like and important properties of powerful $p$-groups is that the minimal number of generators of a subgroup is bounded by the minimal number of generators of the group. We prove an extension of this result for quasi-powerful $p$-groups, and we show that the given bound is close to best possible. Let $d(G)$ denote the minimum number of generators for $G$.

\begin{theorem}
\label{theorem intro bound on rank subgroups}
Let $G$ be a quasi-powerful $p$-group with $d(G)=r$ and let $H \leq G$. Then $d(H) \leq \frac{1}{2}r(r+3)$. 
\end{theorem}

Furthermore, we exhibit an infinite family of examples of groups $G$ such that $d(G)=r$ but with a subgroup $H$ such that $d(H)=\frac{1}{2}r(r+1)$. Hence the bound in Theorem \ref{theorem intro bound on rank subgroups} is close to best possible.

\vspace{2mm}

We now say a few words on the layout of this paper. First, in Section \ref{section preliminiaries} we recall some preliminary results and definitions from the theory of finite $p$-groups. Next, in Section \ref{subsection basic properties of quasi-powerful $p$-groups} we introduce some basic properties of quasi-powerful $p$-groups. We include an example to demonstrate that quasi-powerful $3$-groups are a `new' family with regular power structure. The section culminates with the proof of Theorem \ref{Theorem letter. Quasi powerful order comm less than order elems}. Section \ref{section quasi-powerful pgroups have regular power structure} is split into three parts, with each part corresponding to establishing a part of Theorem \ref{theorem intro regular pow struc and 3 points}.   In Section \ref{subsection an application} we give an application of Theorem \ref{theorem intro regular pow struc and 3 points}, whereby we prove Theorem \ref{ theorem letter more detailed commutator order bound}. Our argument can be adapted to potent $p$-groups (as in Theorem \ref{theorem intro detailed comm bound for potent groups}). Our focus turns to the minimal number of generators of subgroups in Section \ref{section generators for subgroups}. In Section \ref{section on the even prime} we comment on the case when $p=2$. We give an example to show that if the definition of a quasi-powerful $p$-group for odd primes were extended to $p=2$, the groups need not have a regular power structure. 

\vspace{4mm}

\textbf{Notation:} Our notation is standard. We denote the order of $x\in G$ as $o(x)$. In keeping with \cite{Fernandez-Alcober2007}, we use the convention that if $G$ is a $p$-group and $x \in G$, then we define the meaning of the inequality $o(x) \leq p^i$ with $i<0$ to be that $x=1$. We denote the exponent of $G$ by $\exp G$. All iterated commutators are left normed. The terms of the \textit{lower central series} of $G$ are defined recursively as $\gamma_{1}(G)=G$ and $\gamma_{k+1}(G)=[\gamma_{k}(G),G]$ for integers $k\geq 1$. We use bar notation for images in a quotient group; it will always be made explicitly clear what the quotient group under consideration is. We denote the minimal number of generators of a group $G$ by $d(G)$. We denote the Frattini subgroup of $G$ by $\Phi(G)$.  

\section*{Acknowledgements}
I would like to thank Dr Tim Burness for many helpful discussions and for his detailed feedback on earlier versions of this paper. I am also very grateful to Dr Gareth Tracey for his continued encouragement with this project. 

\section{Preliminaries} \label{section preliminiaries}
\setcounter{theorem}{0}
\numberwithin{theorem}{section}
In all of what follows we shall be dealing with finite $p$-groups where $p$ is an odd prime unless explicitly stated otherwise. For the convenience of the reader, we collect here some properties of $p$-groups which shall be used in the rest of the paper. Most of the results are standard, but we draw the reader's attention to Theorem \ref{ Gustavos result for powerful p-groups} and Theorem \ref{Theorem on potent p-groups} which have appeared relatively recently in the literature.
\subsection{Powerful $p$-groups} \label{subsection powerful p group facts}
We recall from \cite{LUBOTZKY1987484} what it means for a $p$-group to be powerfully embedded and for a $p$-group to be powerful.
\begin{definition}
A subgroup $N$ of a finite $p$-group $G$ is \textit{powerfully embedded} in $G$ if $N^{p} \geq [N,G]$ (for $p=2$, if $N^{4} \geq [N,G]$). A finite $p$-group is \textit{powerful} if it is powerfully embedded in itself, that is, if $[G,G] \leq G^{p}$ (for $p=2$, if $[G,G] \leq G^{4}$).
\end{definition}
 
The following theorem demonstrates why powerful $p$-groups are so named - because they are full of $p$th powers. Theorem \ref{theorem intro regular pow struc and 3 points}(\ref{theorem letter group generated by pith powers coincides with set of powers}) and Theorem \ref{theorem letter group generated by pith powers is powerful pgroup} generalise this theorem.
\begin{theorem}[\cite{khukhro_1998}, Theorem 11.10]
\label{theorem properties of pth powers in powerful groups}
Let $G$ be a powerful $p$-group, and let $k \in \mathbb{N}$. 
\begin{enumerate}[(i)]
\item The subgroup $G^{p^{k}}$ coincides with the set $\{x^{p^{k}} | x \in G \}$ of $p^{k}$th powers of elements of $G$; in particular, $(G^{p^{i}})^{p^{j}}=G^{p^{i+j}}$ for all $i,j \in \mathbb{N}$.
\item $G^{p^{k}}$ is powerfully embedded in $G$.
\end{enumerate}

\end{theorem}

\begin{lemma}[\cite{khukhro_1998}, Lemma 11.2]
\label{K is pe in G if  [K,G] lew K^p[K,G,G]}
A normal subgroup $K$ in a finite $p$-group $G$ is powerfully embedded in $G$ if $[K,G] \leq K^{p}[K,G,G]$.
\end{lemma}
Lemma \ref{K is pe in G if  [K,G] lew K^p[K,G,G]} will be used in this paper when we wish to show that certain subgroups are powerfully embedded as it allows for a reduction to a simpler case by assuming $[K,G,G]=1$.

\begin{lemma}[\cite{khukhro_1998}, Lemma 11.7]
\label{powerfully embedded group with an element is powerful}
If $N$ is a powerfully embedded subgroup of $G$, then, for any $h \in G$, the subgroup $H = \langle h \rangle N$ is a powerful $p$-group and $[H,H] \leq N^{p}$. 
\end{lemma}
Lemma \ref{powerfully embedded group with an element is powerful} will enable us to reduce problems about quasi-powerful $p$-groups to powerful $p$-groups. We will show in Proposition \ref{H is powerfully embedded in G} that for a quasi-powerful $p$-group $G$, the subgroup $H=G^{p}Z(G)$ is powerfully embedded in $G$. Hence for any $g\in G$ we will know that $\langle g, H \rangle $ is a powerful $p$-group.

\begin{lemma}[Interchanging Lemma, \cite{SHALEV1993271}, Lemma 3.1]
\label{lemma shalev interchange}
If $M$ and $N$ are powerfully embedded subgroups in a finite $p$-group $P$, then $[M^{p^{i}},N^{p^{j}}]=[M,N]^{p^{i+j}}$ for all $i,j \in \mathbb{N}$.

\end{lemma}

Generalising the next theorem to the case when $G$ is a quasi-powerful $p$-group is one of the main aims of this paper. We remark that we are sometimes able to deploy this theorem directly by reducing from a quasi-powerful group to a powerful subgroup.

\begin{theorem}[\cite{Fernandez-Alcober2007}, Theorem 1]
\label{ Gustavos result for powerful p-groups}
Let $G$ be  a powerful $p$-group. Then, for every $i \geq 0$:

\begin{enumerate}[(i)]
\item If $x,y \in G$ and $o(y) \leq p^{i}$ then $o([x,y]) \leq p^{i}$.
\item If $x,y \in G$ are such that $o(x) \leq p^{i+1}$ and $o(y) \leq p^{i}$, then $o([x^{p^{j}}, y^{p^{k}}]) \leq p^{i-j-k}$ for all $j, k \geq 0$. 
\item If $p$ is odd, then $\exp \Omega_{i}(G) \leq p^{i}$.
\end{enumerate}
\end{theorem}

The following result was originally proved in \cite{L2002}, but an alternate proof was given in \cite{Fernandez-Alcober2007} and \cite{Mazur2007}.

\begin{theorem}[\cite{L2002} Theorem 3.1 and \cite{Fernandez-Alcober2007} Theorem 4]
\label{theorem |G^p^k|=|G: omega k G|}
Let $G$ be a powerful $p$-group. Then $|G:G^{p^{i}}|=|\Omega_{i}(G)|$ for all $i \geq 0$.

\end{theorem}
We will generalise Theorem \ref{theorem |G^p^k|=|G: omega k G|} to the case when $G$ is a quasi-powerful $p$-group in Section \ref{section quasi powerful p groups have a regular power structure}. \par

We now list two of the most abelian-like properties of powerful $p$-groups with respect to generators and subgroups.
\begin{theorem}[\cite{LUBOTZKY1987484}, Theorem 1.12] \label{thm pow p group subgroup rank bounded}
Let $G$ be a powerful $p$-group with $d(G)=r$ and let $H \leq G$. Then $d(H) \leq d(G)$.
\end{theorem}
\begin{theorem}[\cite{LUBOTZKY1987484}, Theorem 1.11]
\label{theorem pow p groups prod of d cyclic groups}
Let $G$ be a powerful $p$-group with $d(G)=d$, then $G$ is a product of $d$ cyclic groups.
\end{theorem}

We obtain variants of both of these results for quasi-powerful $p$-groups in Section \ref{section generators for subgroups}. 

\begin{remark}
\label{abelian likeness of powerful $p$-groups} The properties of powerful $p$-groups given so for demonstrate many of the abelian-like features of powerful $p$-groups. Thus on the one hand we can think of powerful $p$-groups as being close to abelian groups. On the other hand it turns out that we can think of powerful $p$-groups as being close to a typical $p$-group. For example, by a result of Lubotzky it is known that every finite $p$-group appears as a section of a powerful $p$-group (see \cite{Mann2003}, Theorem 1). Additionally it is known that if all characteristic subgroups of a finite $p$-group $G$ can be generated by $r$ elements, then $G$ contains an $r$ generator, powerful, characteristic subgroup whose index is bounded in terms of $p$ and $r$ (\cite{LUBOTZKY1987484}, Theorem 1.14). 
\end{remark}

We close this discussion on powerful $p$-groups by recalling a result from \cite{GONZALEZSANCHEZ2004193}. First we need the following definition. 
\begin{definition}
\label{defn potent pgroup}
A finite $p$-group is \textit{potent} if $[G,G]\leq G^{4}$ for $p=2$ or $\gamma_{p-1}(G) \leq G^{p}$ for $p>2$.
\end{definition}
Note that for $p\in \{ 2,3 \}$, the definitions of potent and powerful coincide.

The next result will be a crucial ingredient in the proof of Theorem \ref{theorem intro regular pow struc and 3 points}(\ref{thm intro: product of two elements of order at most p^i has order at most p^i}).

\begin{theorem}[\cite{GONZALEZSANCHEZ2004193}, Theorem 5.1]
\label{Theorem on potent p-groups}
Let $G$ be a powerful $p$-group and $N$ a normal subgroup of $G$. Then one of the following two properties holds:
\begin{enumerate}[(i)]
\item For any $i,s,t \geq 0$ such that $n = i+s+t \geq 1$ if $p$ is odd, $[G,G]^{p^{n}} \leq [N^{p^{i}}, G^{p^{s}}]^{p^{t}}$.
\item There exists a proper powerful subgroup $T$ of $G$ such that $N \leq T$.
\end{enumerate}

\end{theorem}

\subsection{Regular $p$-groups}
\label{subsection regular p groups}
Regular $p$-groups were introduced by P. Hall in his pioneering paper \cite{phallcontribtothetheoryofgroupsofprimepowerorder}. 
\begin{definition}
\label{defn of regular p group}
Let $G$ be a finite $p$-group. We say $G$ is a \textit{regular} $p$-group if for every $x,y \in G$ we have that $x^{p}y^{p} = (xy)^{p}c$ for some $c \in \gamma_{2}(\langle x,y \rangle)^{p} $.
\end{definition}
We now recall two results from the theory of regular $p$-groups. The first gives a condition on when a group is regular, based on the nilpotency class of the group. The second result tells us that in a regular $p$-group the order of a product of two elements cannot exceed the order of the factors. 

\begin{theorem}[ \cite{phallcontribtothetheoryofgroupsofprimepowerorder}, Corollary 4.14, Theorem 4.26]
\label{properties of regular pgroups}
Let $G$ be a finite $p$-group.
\begin{enumerate}[(i)]
\item If the nilpotency class of $G$ is less than $p$ then $G$ is regular.
\item If $a$ and $b$ are any two elements of the regular $p$-group $G$, then the order of $ab$ cannot exceed the orders of both $a$ and $b$. In particular for any $i \geq 0$ the subgroup $\Omega_{i}(G)= \{ x \in G \mid x^{p^{i}}=1     \}.$ 
\end{enumerate}

\end{theorem}
One of the main results of this paper is a version of Theorem \ref{properties of regular pgroups} (ii) for quasi-powerful $p$-groups. For $p$ sufficiently large we will see that we can use Theorem \ref{properties of regular pgroups} (i) to reduce to the case that the group is regular. However a different argument is needed for $p=3$.

\section{Basic Properties of Quasi-powerful $p$-groups}
\numberwithin{theorem}{section}

\label{subsection basic properties of quasi-powerful $p$-groups}
In this section we introduce the basic properties of quasi-powerful $p$-groups for odd primes $p$. The results proved in this section will be used throughout the rest of the paper. Proposition \ref{H is powerfully embedded in G} will sometimes allow us to reduce to a powerful subgroup within a quasi-powerful group. At the end of this section this idea is used to prove Theorem \ref{Theorem letter. Quasi powerful order comm less than order elems}. 

\begin{definition}
Let $p$ be an odd prime, we say that a $p$-group $G$ is a \textit{quasi-powerful} $p$-group if $G/Z(G)$ is a powerful $p$-group.
\end{definition}

We do not give a definition of quasi-powerful groups for $p=2$, but in Section \ref{section on the even prime} we give an example which suggests for $p=2$ a different definition is needed. 

Suppose that $G/Z(G)$ is powerful. Throughout the rest of this paper, we shall let $H=G^{p}Z(G)$.  Notice that $G^{\prime} \leq H$. From this it is clear that powerful $p$-groups and groups of nilpotency class $2$ are quasi-powerful $p$-groups. However there exist quasi-powerful $p$-groups which are neither of those things. We now give an example of a $3$-group of nilpotency class $3$ such that $G/Z(G)$ is powerful, but $G$ is neither regular nor powerful. 
\begin{example}
\label{example of a quasipowerful group}
 Let $G$ be the group with presentation: 
$$\langle a,b,c,d \mid a^{27}, b^{3}, c^{27}, d^{3},a^b=a, a^c=a^4b, a^d=a, b^c=ba^9, b^d=b, c^d=cb^{-1} \rangle. $$
The following details are easily checked using GAP \cite{GAP4.8.4}, where this group can be constructed as \texttt{SmallGroup(6561,86718)} using the package \texttt{SglPPow} \cite{SglPPow1.1}. We can describe the structure of this group as $$ ((\mathbb{Z}_{27} \times \mathbb{Z}_{3}) \rtimes \mathbb{Z}_{27}) \rtimes \mathbb{Z}_{3}. $$
In addition, we have $Z(G)=\langle a^{3}b^{-1}, a^{9}, c^{9} \rangle$ and $G^{\prime}=\langle a^{3}, b \rangle$. The group is not powerful because $b \notin G^{3}$. However $G/Z(G)$ is powerful, since $b \in Z(G)G^{p}$. Furthermore since for $p=3$ the definitions of potent and powerful coincide, $G$ is also not a potent $p$-group. \par
Moreover one can show that this group is not a regular $p$-group. For example let $x=a^{18}c^{18}d$ and $y=c$. Then $$(xy)^{-3}\cdot x^3 \cdot y^3 \notin \left(\gamma_{2}(\langle x,y \rangle)\right)^{3}.$$ 
\end{example}
Example \ref{example of a quasipowerful group} demonstrates a quasi-powerful $p$-group  which does not fall into one of the families which are already known to have a regular power structure.  

Before moving on we make the following remark which will be used frequently throughout the rest of this paper.
\begin{remark}
It is easy to see that the property of being a quasi-powerful $p$-group is preserved under taking quotients. However it is not necessarily preserved under taking subgroups. For instance in Example \ref{example of a quasipowerful group}, one can check that the subgroup $\langle a^{3}, b, c, d \rangle $ is not quasi-powerful.
\end{remark}

We now begin by investigating the subgroup $H=G^{p}Z(G)$.

\begin{lemma}
\label{lemma f^ph^p=j^pz}
For any $g,h \in G$ we have that $g^{p}h^{p}=j^{p}z$ for some $j \in G$ and $ z \in Z(G)$. 
\end{lemma}
\begin{proof}
Since $\bar{G}=G/Z(G)$ is powerful, in the quotient group the product of $p$th powers is a $p$th power and so $ g^{p}Z(G) \cdot h^{p}Z(G)= j^{p} Z(G)$ for some $j \in G$. Thus $g^{p}h^{p}=j^{p}z$ for some $j \in G$ and $z \in Z(G)$. 
\end{proof}

\begin{remark} \label{formofelementsinH}
This means that any $h \in H$ is of the form $g^{p}z$ for some $g\in G$ and $z \in Z(G)$, since $H=G^{p}Z(G)$, so $h= g_{1}^{p} \dots g_{t}^{p} z_{1} = g^{p}z$ by repeated application of Lemma \ref{lemma f^ph^p=j^pz}.
\end{remark}

\begin{proposition}
\label{H is powerfully embedded in G}
The subgroup $H$ is powerfully embedded in $G$.
\end{proposition}

\begin{proof}
By Lemma \ref{K is pe in G if  [K,G] lew K^p[K,G,G]}, we may assume that $[H,G,G]=1$. Consider some $h \in H$. By Remark \ref{formofelementsinH} we can write $h=x^{p}z$ for some $x\in G$ and $z \in Z(G)$. Let $ g \in G$ and consider
\begin{align}
    [h,g] &= [x^{p}z,g] \nonumber \\
    &= [x^{p},g] \nonumber \\
    &= [x,g]^{x^{p-1}}[x,g]^{x^{p-2}} \dots [x,g] \nonumber \\
    &= [x,g][x,g,x^{p-1}][x,g][x,g,x^{p-2}]\dots [x,g]. \label{eqndagger}
\end{align}
We have $p$ of the $[x,g]$ terms and $p-1$ of the $[x,g,\star]$ terms. Observe that $[x,g] \in G^{\prime} \leq H $. Hence these terms of weight $3$ are central because $[H,G,G]=1$. Also notice that this implies $[x,g,x^{i}]=[x,g,x]^{i}$. Hence (\ref{eqndagger}) becomes 
$$[x,g]^p[x,g,x]^{(p-1)+(p-2)+\dots +1} = [x,g]^p[x,g,x]^{1/2 (p-1)p}.$$ 
As $[x,g]\in H$ and $[x,g,x]\in H$ and $p$ is an odd prime we see that $[h,g]\in H^{p}$. It follows that $[H,G] \leq H^p$. \end{proof}
\begin{remark}
As $H$ is powerfully embedded in $G$, we can conclude that for $p>3$ the group $G$ is potent, since $$\gamma_{p-1}(G) \leq [H,\underbrace{G,\dots, G}_{p-3}]\leq H^{p} \leq G^{p}.$$ Thus if $p>3$  we could appeal to \cite{GONZALEZSANCHEZ2004193} to conclude that these groups have a regular power structure. However, we will give an independent proof of this fact and deal with all odd primes. Furthermore we will see the most involved and interesting case is when $p=3$, and we have already seen that a quasi-powerful $3$-group need not be potent.
\end{remark}

We now prove an analogue of Theorem \ref{ Gustavos result for powerful p-groups}(i) for quasi-powerful $p$-groups, this is Theorem \ref{Theorem letter. Quasi powerful order comm less than order elems}.

\begin{proof}[Proof of Theorem \ref{Theorem letter. Quasi powerful order comm less than order elems}]
Let $g,h \in G$ with $g$ of order at most $p^{i}$. The commutator $[g,h]$ can be written as $g^{-1} g^{h}$ where both terms of the product are of order at most $p^{i}$. The group $W= \langle g, H \rangle $ contains $[g,h]$ and $g^{-1}$ and so contains $g^{h}$. The group $W$ is powerful by Lemma \ref{powerfully embedded group with an element is powerful} and Proposition \ref{H is powerfully embedded in G}. Then by Theorem \ref{ Gustavos result for powerful p-groups} (i) we have that $o(g^{-1}\cdot g^{h})$ is at most $p^{i}$. That is $o([g,h]) \leq p^{i}$.
\end{proof}
\begin{remark}
\label{remark if Gprime is in a pow embedded subgroup then bound on comm}
In fact this property is true for any group $G$ containing a powerfully embedded subgroup $N$ with $G^{\prime} \leq N$.
\end{remark}

Theorem \ref{Theorem letter. Quasi powerful order comm less than order elems} will be used frequently throughout the remainder of the paper.
\section{Quasi-powerful $p$-groups have regular power structure}
\label{section quasi-powerful pgroups have regular power structure}
\numberwithin{theorem}{section}

\subsection{The exponent of omega subgroups}
\label{subsection proving omega condition}
In this section our aim is to prove that $\exp \Omega_{i}(G) \leq p^{i}$.
This can be stated equivalently as Theorem \ref{theorem intro regular pow struc and 3 points} (\ref{thm intro: product of two elements of order at most p^i has order at most p^i}).

We will see that by using properties of powerful $p$-groups and regular $p$-groups we can relatively quickly deal with the case of primes greater than $3$. However when $p=3$ we can no longer assume the subgroup $\Omega_{1}(G)$ is regular and the situation becomes more involved.
\begin{lemma}
\label{lemma x,y,z,w with x,y,z order p then comm is trivial}
If $x,y,z,w \in G$, where $x,y,z$ are elements of order $p$, then $$[x,y,z,w]=1.$$ 
\end{lemma}

\begin{proof}
We use bar notation to denote the image of an element in the quotient group, $\bar{G}=G/Z(G)$. In particular as $\bar{G}$ is powerful and $\bar{x}, \bar{y}, \bar{z}$ are of order at most $p$ then $[\bar{x},\bar{y}]= \bar{g}^{p}$ for some $\bar{g} \in \bar{G}$ with $\bar{g}$ of order at most $p^{2}$ by Theorem \ref{ Gustavos result for powerful p-groups} (i). Then by Theorem \ref{ Gustavos result for powerful p-groups} (ii) setting $i=1$ it follows that $[\bar{x},\bar{y},\bar{z}]=[\bar{g}^p,\bar{z}]=\bar{1}$. Lifting back up to $G$, this means that $[x,y,z] \in Z(G)$ and thus $[x,y,z,w]=1$.
\end{proof}

\begin{lemma}
\label{Nilpotency class of omega1 is at most 3}
The nilpotency class of $ \Omega_{1}(G)$ is at most $3$. 
\end{lemma}
\begin{proof}
We note that $\Omega_{1}(G)$ is generated by the elements in $G$ of order $p$, say $a_{1}, \dots , a_{k}$. By Lemma \ref{lemma x,y,z,w with x,y,z order p then comm is trivial} it is clear that any commutator of weight $4$ in these must be trivial.
\end{proof}
\begin{proposition}
\label{prop omega_1 has exponent at most p}
If  $p$ is a prime such that $p>3$, then $\Omega_{1}(G)$ has exponent at most $p$.
\end{proposition}
\begin{proof}
As $p$ is a prime greater than $3$, Lemma \ref{Nilpotency class of omega1 is at most 3} implies that $p$ is greater than the nilpotency class of $\Omega_{1}(G)$. Then Theorem \ref{properties of regular pgroups} implies that $\Omega_{1}(G)$ is regular and in particular the product of two elements of order $p$ has order at most $p$. The result follows.
\end{proof}

We now begin to deal with the difficult case of $p=3$. We will eventually see that we can reduce to the case where $G$ is a quasi-powerful $3$-group with cyclic centre and $\Omega_{1}(G)^{p} \leq Z(G)$.

\begin{lemma}
\label{ 3-group with class at most 3 has exp omega as 3}
Let $G$ be a finite $3$-group  of nilpotency class at most $3$  such that $G/Z(G)$ is powerful. Then $\exp \Omega_{1}(G) \leq 3$. 
\end{lemma}
\begin{proof}
Let $a,b \in G$ both be of order $3$. We expand $(ab)^3$, making use of the fact that the nilpotency class is at most $3$. 
$$(ab)^3=a^3 b^3 [b,a]^3 [b,a,b]^5[b,a,a].$$
By Theorem \ref{Theorem letter. Quasi powerful order comm less than order elems} we know that any commutator containing $a$ or $b$ has order at most $3$. Hence we  obtain:
\begin{align*}
    (ab)^3 &= [b,a,b]^2[b,a,a] = [b,[b,a]][[b,a],a].
\end{align*}
As $G/Z(G)$ is powerful, we can write $[b,a]=g^3 z$ for some $g \in G$ and $z \in Z(G)$. Then
\begin{align*}
[b,[b,a]] &= [b,g^3 z] \\
          &= [b,g^3] \\
          &=[b,g][b,g]^{g}[b,g]^{g^{2}} \\
          &= [b,g][b,g][b,g,g][b,g][b,g,g^{2}] \\
          &= [b,g]^3[b,g,g]^3 \\
          &= 1,
\end{align*}
where again we use that the nilpotency class is at most $3$ and the fact that any commutator containing the element $b$ has order at most $3$.  That the commutator $[[b,a],a]$ is trivial follows in a similar way.
Thus we can conclude that if $a$ and $b$ have order at most $3$ then so does $ab$ and hence $\exp \Omega_{1}(G) \leq 3$.
\end{proof}
Next we  explain how we can make the reduction to the case where the centre of our quasi-powerful $3$-group $G$ is cyclic and that $\Omega_{1}(G)^{p} \leq Z(G)$ in our goal to show that $\exp \Omega_{1}(G) \leq 3$. 

Suppose that a $3$-group $G$ is a quasi-powerful group of smallest order such that the exponent of $\Omega_{1}(G)$ is greater than $3$. In this case there  must exist elements $a$ and $b$ in $G$ both of order $3$ such that $o(ab)>3$. Let $N \leq Z(G) $ be a subgroup of order $p$. Then $G/N$ is a quasi-powerful group of smaller order, and so in this group we must have that $ab N$ has order at most $3$, in other words $(ab)^{3} \in N$. This allows us to assume that $\Omega_{1}(G)^{3} \leq Z(G)$. Furthermore, if the centre of $G$ is not cyclic then it would contain two distinct subgroups $N_{1}$ and $N_{2}$ of order $3$ such that $N_{1} \cap N_{2} = 1$. Then we would be able to conclude that $(ab)^3$ were in both, and thus $(ab)^{3}=1$. Thus for what follows we consider a quasi-powerful $3$-group $G$ of minimal order such that $\exp \Omega_{1}(G)>3$ and therefore we may assume the centre is cyclic and that $\Omega_{1}(G)^{3} \leq Z(G)$.
\begin{proposition}
\label{prop exponent of omega 1 for p=3}
Let $G$ be a quasi-powerful $3$-group. Then the exponent of $ \Omega_{1}(G)$ is at most $3$.
\end{proposition}
\begin{proof}
By the discussion above, we can assume that $G$ has cyclic centre and that  $\Omega_{1}(G)^{p} \leq Z(G)$. 

Consider the subgroup $J=\Omega_{1}(G)Z(G)$. Notice that $J$ is normal (in fact characteristic) in $G$. We use bar notation to denote images under the natural map corresponding to quotienting by $Z(G)$. Consider the image, $\bar{J}$ of $J$ in $\bar{G}=G/Z(G)$. The subgroup $\bar{J}$ is normal in  the powerful group $\bar{G}$.  Then by Theorem \ref{Theorem on potent p-groups} we have two cases to consider.

In the first case, we can assume that we have $[\bar{G}, \bar{G}]^{p} \leq [\bar{J}^{p}, \bar{G}]$. Notice that $\bar{J}^{p}=\bar{1}$, since $J^{p} \leq Z(G)$. Hence $[\bar{G},\bar{G}]^{p}=\bar{1}$. Now, using the Interchanging Lemma \ref{lemma shalev interchange}, since $\bar{G}$ is powerful, we have that $$[\bar{G}, \bar{G}, \bar{G}] \leq [\bar{G}^{p}, \bar{G}] = [\bar{G}, \bar{G}]^{p} = 1. $$
Hence $\bar{G}$ has class at most $2$, and so $G$ has class at most $3$. Then by Lemma \ref{ 3-group with class at most 3 has exp omega as 3} we obtain the desired result.

In the second case we can assume that $\bar{J}$ is contained in some proper  powerful subgroup of $\bar{G}$. Call this subgroup $\bar{P}$.  Then lifting up to the proper subgroup $P$ of $G$ where $\Omega_{1}(G) \leq P$, we see that  
$$[P,P] \leq P^{3}(Z(G) \cap P) \leq P^{3}Z(P). $$ 
Hence $P$ is a quasi-powerful group of order strictly less than $G$. Therefore the exponent of $\Omega_{1}(P)$ is at most $3$. In particular, the product of two elements of order $3$ in $P$ has order at most $3$. Then since $\Omega_{1}(G) \leq P$ we conclude that the exponent of $\Omega_{1}(G)$ is at most $3$.
\end{proof}
Thus we see that for any odd prime $p$ we have that $\exp \Omega_{1}(G) \leq p$. This now enables us to prove the more general result that  $\exp \Omega_{i}(G) \leq p^{i}$ for any $i$ and any odd prime $p$ - this is Theorem \ref{theorem intro regular pow struc and 3 points}(\ref{thm intro: product of two elements of order at most p^i has order at most p^i}).

\begin{proof}[Proof of Theorem \ref{theorem intro regular pow struc and 3 points} (\ref{thm intro: product of two elements of order at most p^i has order at most p^i})]
Let $G$ be a quasi-powerful $p$-group and $i\geq 0$. We proceed by induction on the  order of $G$. The result is clearly true for groups of order $p^{0}$. Now suppose that $|G| \geq p$ and that the claim holds for all quasi-powerful $p$-groups of smaller order.  Notice that if $a\in G$ with $a$ having order  $p^{j}\leq p^{i},$ then $a^{p^{j-1}} \in \Omega_{1}(G)$ and thus $a\Omega_{1}(G)$ has order at most $p^{j-1}\leq {p^{i-1}}$ in $\frac{\Omega_{i}(G)}{\Omega_{1}(G)}$. It follows that  $$\frac{\Omega_{i}(G)}{\Omega_{1}(G)} \leq \Omega_{i-1}\left(\frac{G}{\Omega_{1}(G)}\right).$$
Then by the inductive hypothesis, since $\frac{G}{\Omega_{1}(G)}$ is a quasi-powerful $p$-group of smaller order, we have that the exponent of $ \Omega_{i-1}(\frac{G}{\Omega_{1}(G)}) $ is at most $p^{i-1}$ and consequently the subgroup $\frac{\Omega_{i}(G)}{\Omega_{1}(G)}$ has exponent at most $p^{i-1}$. Then lifting back up to $G$ we see that $(\Omega_{i}(G))^{p^{i-1}} \leq \Omega_{1}(G)$ and by Proposition \ref{prop omega_1 has exponent at most p} and Proposition \ref{prop exponent of omega 1 for p=3} it follows that    $\Omega_{i}(G)^{p^{i}}=1$.
\end{proof}

\subsection{$p$th powers}
\label{subsection on pth powers in quasipowerful pgroups}
In this section we prove properties about the groups of $p^{i}$th powers in quasi-powerful $p$-groups. We will prove  Theorem \ref{theorem intro regular pow struc and 3 points} (\ref{theorem letter group generated by pith powers coincides with set of powers}) and Theorem \ref{theorem letter group generated by pith powers is powerful pgroup}. 

We first prove that just like in a powerful  $p$-group, the product of $p$th powers in a quasi-powerful $p$-group is equal to a $p$th power. This is the first step in proving Theorem \ref{theorem letter group generated by pith powers coincides with set of powers}.

We will need to recall the following formulation of the collection formula of Philip Hall (see \cite{mckay2000finite}, Exercise 1.2). If $G$ is a group, $x,y \in G,$ and $n \in \mathbb{N}$ then  
\begin{equation}
\label{p hall collection  (xy)^p^n}
(xy)^{p^{n}} \equiv x^{p^{n}}y^{p^{n}} \left(\thinspace \text{mod} \thinspace \gamma_{2}(T)^{p^{n}}\gamma_{p}(T)^{p^{n-1}}\dots \gamma_{p^{n}}(T)\right)
\end{equation} where $T= \langle x,y \rangle$.
\begin{theorem}
\label{G^p is set of pth powers}
Let $G$ be a quasi-powerful $p$-group and $p$ an odd prime. Then $$G^{p} = \{g^{p} \mid g \in G\}.$$
\end{theorem}
\begin{proof}
Let $a,b \in G$. We recall that $H= G^{p} Z(G)$. 
Using the collection formula (\ref{p hall collection  (xy)^p^n}) we have
$$ (ab)^{p} = a^{p} b^{p} \gamma_{2}(T)^{p} \gamma_{p}(T) $$
where $T = \langle a, b \rangle $.
Now notice that $\gamma_{2}(T) \leq \gamma_{2}(G) \leq H$ and so $\gamma_{2}(T)^{p} \leq H^{p}$. Also notice that since $p\geq 3$ we have that $\gamma_{p}(T) \leq [T, T, T] \leq [H,T] \leq H^{p}$ because $H$ is powerfully embedded by Proposition \ref{H is powerfully embedded in G}. Now since $H$ is powerful by Proposition \ref{H is powerfully embedded in G}, $H^{p}$ contains precisely the $p$th powers of elements of $H$. Thus we have that 
$$(ab)^{p} = a^{p} b^{p} h_{1}^{p}$$ for some $h_{1} \in H$. Then $(ab)^{p}\cdot h_{1}^{-p}=a^{p}b^{p}$.
Now let $W = \langle (ab), H \rangle $. By Lemma \ref{powerfully embedded group with an element is powerful} and Proposition \ref{H is powerfully embedded in G}, it follows that $W$ is powerful and so $$ (ab)^{p}\cdot h_{1}^{-p} = w^{p}$$ for some $w \in W$. Thus $a^{p}b^{p}=w^{p}$.
\end{proof}

In \cite[p.142]{L2002} and \cite[Question 5.2.4]{wilsonsphd} L. Wilson raises the following question:

\begin{question}
\textit{If $G$ is a $p$-group (with $p$ odd) and the $p$th powers of elements of $G$ form a subgroup, must this subgroup be powerful?}
\end{question}

Wilson argues how an affirmative answer to this question would in turn provide an affirmative answer to a question of A. Shalev \cite[Problem 13]{Shalev1995}.
\begin{question} \label{question of shalev}
\textit{Let $G$ be a finitely generated pro-$p$ group, and suppose that for each $x,y \in G$ there is $z \in G$ such that $x^{p}y^{p}=z^{p}$; does it follow that $G$ is $p$-adic analytic?}
\end{question}
This provides some motivation for our next results, as in light of this question it is natural to ask whether $G^{p}$ must be powerful for quasi-powerful $p$-groups. We will need the following version of Philip Hall's commutator expansion formula (see \cite{mckay2000finite}, Exercise 1.2). Let $G$ be a group, $x,y \in G$ and $n \in \mathbb{N}$, then 
\begin{equation}
\label{eqn halls commutator expansion formula [x,y]^p^n}
   [x,y]^{p^{n}} \equiv [x^{p^{n}},y] \left( \thinspace \text{mod} \thinspace \gamma_{2}(M)^{p^{n}} \gamma_{p}(M)^{p^{n-1}} \dots \gamma_{p^{n}}(M) \right)  
\end{equation}
where $M=\langle x, [x,y] \rangle .$ 

\begin{theorem}
\label{theorem G^p is powerful}
Let $p$ be an odd prime. If $G$ is a quasi-powerful $p$-group then $G^{p}$ is powerful. 
\end{theorem}

\begin{proof}
We wish to show that $[G^{p},G^{p}] \leq (G^{p})^p$. By Lemma \ref{K is pe in G if  [K,G] lew K^p[K,G,G]} we may assume that $[G^{p}, G^{p}, G^{p}]=1$. Also as $(G^{p})^{p^{2}} \leq (G^{p})^{p}$ we can quotient by $(G^{p})^{p^{2}}$ and in particular can assume that $G$ has exponent at most $p^{3}$.
Consider $x,y \in G$. We will show that $[x^p, y^p] = [x, y^p]^p$. Using the collection formula (\ref{eqn halls commutator expansion formula [x,y]^p^n}) we have: 
$$[x^p, y^p] \equiv [x, y^p]^{p} \thinspace \text{mod} \thinspace \gamma_{2}(M)^p \gamma_{p}(M) $$
where $M = \langle x, [x,y^p] \rangle $.
We need to show that $\gamma_{2}(M)^p = 1$ and $\gamma_{p}(M)=1$.
First we show that $\gamma_{2}(M)^{p}=1$. 
\par Recall that by Theorem \ref{theorem intro regular pow struc and 3 points}(\ref{thm intro: product of two elements of order at most p^i has order at most p^i}), the product of elements of order $p$ has order at most $p$, and therefore we only need to show that the generators of $\gamma_{2}(M)$ have order at most $p$ to be able to conclude that all elements have order at most $p$. Notice that $[x,y^{p}]=g^{p^{2}}z$ for some $g \in G$ and $z \in Z(G)$. As the exponent of $G$ is at most $p^{3}$ we can assume that $g^{p^{2}}$ has order at most $p$, and then by Theorem \ref{Theorem letter. Quasi powerful order comm less than order elems} it follows that any commutator which includes the element $g^{p^{2}}$ as a term, has order at most $p$. Therefore it follows that every element in $\gamma_{2}(M)$ has order at most $p$ and so $\gamma_{2}(M)^{p} = 1$.

Next we show that $\gamma_{p}(M)=1$. Observe that because $\bar{G}=G/Z(G)$ is powerful and of exponent at most $p^{3}$ we must have that $[G^{p^{2}},G] \leq Z(G)$. Using again the fact that we can write $[x,y^{p}]=g^{p^{2}}z$ and that $p\geq 3$ we have that $\gamma_{p}(M) \leq[G^{p^{2}},G,G]=1$. 

Thus we have that $[x^{p},y^{p}]=([x,y^{p}])^{p}$ and $[x,y^{p}]\in G^{p}$ since $G^{p}$ is normal and so $[x^{p},y^{p}] \in (G^{p})^{p}$ as required.
\end{proof}
\begin{remark}
Although $G^{p}$ is powerful, $G^{p}$ may not be powerfully embedded in $G$ (see Remark \ref{endnote to show quasi-powerful pgroups need not have Gp pow embedded}).
\end{remark}
We are now able to apply the results of this section to prove Theorem \ref{theorem intro regular pow struc and 3 points}(\ref{theorem letter group generated by pith powers coincides with set of powers}) and Theorem \ref{theorem letter group generated by pith powers is powerful pgroup}.
\begin{proof}[Proof of Theorem \ref{theorem intro regular pow struc and 3 points}(\ref{theorem letter group generated by pith powers coincides with set of powers})]
For $k=1$ the claim follows from Theorem \ref{G^p is set of pth powers}. For $k \geq 2$ observe that by Theorem \ref{theorem G^p is powerful} we know that $G^{p}$ is powerful, and then by Theorem \ref{theorem properties of pth powers in powerful groups} (i) we know that for a powerful group the set of $p^{k-1}$th powers and the group generated by $p^{k-1}$th powers coincide. In particular we have that 
\begin{align*}
   G^{p^{k}} &= \langle g^{p^{k}} \mid g \in G \rangle \\
             &= \langle x^{p^{k-1}} \mid x \in G^{p} \rangle \\ &=(G^{p})^{p^{k-1}} \\
             &=\{ x^{p^{k-1}} \mid x \in G^{p} \} \\
             &= \{ x^{p^{k}} \mid x \in G \} 
\end{align*}
\end{proof}

\begin{proof}[Proof of Theorem \ref{theorem letter group generated by pith powers is powerful pgroup}]
By Theorem \ref{theorem G^p is powerful} we know that $G^{p}$ is powerful, and so by standard properties of powerful $p$-groups it follows that $(G^{p})^{p^{i}}$ is powerful for all $i \geq 0$. By Theorem \ref{theorem intro regular pow struc and 3 points}(\ref{theorem letter group generated by pith powers coincides with set of powers}) we know that $(G^{p})^{p^{i}}=G^{p^{i+1}}$.
\end{proof}

\subsection{The index of agemo subgroups}
\label{section quasi powerful p groups have a regular power structure}
We now move on to proving the final condition to show that quasi-powerful $p$-groups have a regular power structure. Our arguments in this section rely heavily on the results of L. Wilson in \cite{L2002}. \par
Recall that a $p$-group $G$ is said to have  a \textit{regular power structure} if the following three conditions hold for all positive integers $i$:
\begin{align}
 \tag{\ref{eqn pth power}}G^{p^{i}} &=\{g^{p^{i}} \mid g \in G \}.  \\ 
 \tag{\ref{eqn omega conditon}}\Omega_{i}(G) &= \{g\in G \mid o(g) \leq p^{i} \}.  \\ 
 \tag{\ref{eqn index condition}}|G:G^{p^{i}}| &= | \Omega_{i}(G)|.
\end{align}
The first and second conditions have been established. Thus all that remains is to prove the final condition, that $[G:G^{p^{k}}]=|\Omega_{k}(G)|$. To begin we prove a base case, when $k=1$.

\begin{proposition}
\label{prop base case G G^p |omega 1 G|}
Let $G$ be a quasi-powerful $p$-group, then $[G:G^{p}]=|\Omega_{1}(G)|$.
\end{proposition}

\begin{proof}
Suppose the result holds for all quasi-powerful $p$-groups of smaller order. We consider two cases depending on the exponent of $Z(G)$. For the first case, suppose that $Z(G)$ has exponent $p$. Here we can assume there must be an element $x \in Z(G)$ of order $p,$ with $x \notin G^{p},$ otherwise $G$ would be powerful and we are done. 
Let $N = \langle x \rangle $ and $\bar{G}=G/N$. Suppose $|G|=p^{n}$. Notice first that $|G^{p}|=|\bar{G}^{p}|,$ since $N \nleq G^{p}.$ We need to consider $ \Omega_{1}(\bar{G}).$ We have $\Omega_{1}(G)/N \leq \Omega_{1}(\bar{G})$. If $x \in \Omega_{1}(\bar{G})$ but $x \notin \Omega_{1}(G)$ it would mean that $x^{p} \in N \setminus \{1\}$ but then $N \leq G^{p}$, a contradiction. Thus we must have $\Omega_{1}(G)/N = \Omega_{1}(\bar{G})$. Then $|\bar{G}:\bar{G}^{p}|=|\Omega_{1}(\bar{G})|$, that is 
$$\frac{|G|}{|N||G^{p}|}=\frac{|\Omega_{1}(G)|}{|N|} $$ and the result follows. 

Now for the second case, suppose that the exponent of $Z(G)$ is greater than $p$. In this case we can find some element $z$ of order $p^{2}$ such that $z$ is central in $G$. Let $N=\langle z^{p}\rangle$.  Notice that $$|G:G^{p}|=|\bar{G}:\bar{G}^{p}|.$$ Thus if we can show that $|\Omega_{1}(G)|=|\Omega_{1}(\bar{G})|$ then we are done. Observe that $z^{p}$ is of order $p$ and so $z^{p} \in \Omega_{1}(G)$. Consider the cosets $a_{1}N, \dots, a_{t}N$ of $N$ in $\Omega_{1}(G)$, where $t=\frac{|\Omega_{1}(G)|}{|N|}$. Then we can assume that $\bar{a_{1}}, \dots, \bar{a_{t}} \in \Omega_{1}(\bar{G})$, and so $|\Omega_{1}(\bar{G})| \geq \frac{|\Omega_{1}(G)|}{|p|}$. Next observe that $z$ is of order $p^{2}$ in $G$ and so $z \notin \Omega_{1}(G)$ but $\bar{z}$, the image of $z$ in $G/N$, has order $p$ and so $\bar{z} \in  \Omega_{1}(\bar{G})$. Since  $\bar{z}$ is central, it is easy to see that $\langle a_{1}, \dots, a_{t}, \bar{z} \rangle $ has order $p \cdot \frac{|\Omega_{1}(G)|}{p}=|\Omega_{1}(G)|$. We now show that everything in $\Omega_{1}(\bar{G})$ is accounted for. Suppose that $\bar{x} \in \Omega_{1}(\bar{G})$ but $\bar{x} \notin \frac{\Omega_{1}(G)\langle z \rangle }{N}$. Then $x^{p} = z^{\lambda p}$ for some $0< \lambda < p$, but then $(xz^{-\lambda})^{p}=x^{p}z^{-\lambda p}=1$ and so $xz^{-\lambda} \in \Omega_{1}(G)$ but then $\bar{x} \in \frac{\Omega_{1}(G)\langle z \rangle }{N}$. Hence we can conclude that $|\Omega_{1}(G)|=|\Omega_{1}(\bar{G})|$ and the result follows.
\end{proof}

We now prove the more general result by induction. We remark that our proof below follows Wilson's proof of Theorem 3.1 in \cite{L2002} very closely. Moreover, we will call upon the following result of Wilson taken from \cite{L2002}. We will let $\mathcal{O}_{p}$ be the class of all $p$-groups for which $\Omega_{k}(G)$ is the set of elements of order dividing $p^{k}$ for all $k$.
\begin{lemma}[\cite{L2002}, Lemma 2.1]
\label{lemma If G is in O_k then we have a property about quotients by omega}
Let $G$ be in $\mathcal{O}_{p}$. Then for all $m$ and $k$ 
$$ \Omega_{k}(G/\Omega_{m}(G)) = \Omega_{m+k}(G)/\Omega_{m}(G).$$
\end{lemma}
We are now in a position to establish Theorem \ref{theorem intro regular pow struc and 3 points}(\ref{theorem letter regular power structure last condition on order of G}).

\begin{proof}[Proof of Theorem \ref{theorem intro regular pow struc and 3 points}(\ref{theorem letter regular power structure last condition on order of G})]
We use induction on $k$. We have established the base case in Proposition \ref{prop base case G G^p |omega 1 G|}. Assume now that the result holds for $k$. We wish to find the order of $G^{p^{k+1}}$. By Theorem \ref{theorem intro regular pow struc and 3 points} (\ref{theorem letter group generated by pith powers coincides with set of powers}) we have that $G^{p^{k+1}}=(G^{p})^{p^{k}}.$ As $G^{p}$ is powerful,  we can apply Theorem \ref{theorem |G^p^k|=|G: omega k G|} and conclude that 
\begin{equation}
\label{eqn G^p^k+1 = G^p:...} 
    |G^{p^{k+1}}|=|G^{p}:\Omega_{k}(G^{p})|
\end{equation}

By Theorem \ref{theorem intro regular pow struc and 3 points}(\ref{thm intro: product of two elements of order at most p^i has order at most p^i}). we know that the exponent of $\Omega_{k}(G)$ is at most $p^{k}$ and so we have $\Omega_{k}(G^{p})=\Omega_{k}(G) \cap G^{p}.$ Then $$ G^{p}/\Omega_{k}(G^{p}) \cong G^{p}\Omega_{k}(G)/\Omega_{k}(G) = (G/\Omega_{k}(G))^{p}.$$
Thus 
\begin{equation}
\label{eqn G^p: omega...}
|G^{p} : \Omega_{k}(G^{p})| = |(G/\Omega_{k}(G))^{p}|.
\end{equation}

Now since quotients of quasi-powerful $p$-groups are quasi-powerful $p$-groups, we have that $G/\Omega_{k}(G)$ is a quasi-powerful $p$-group. We can apply the base case, Proposition \ref{prop base case G G^p |omega 1 G|}, to find that $$|(G/\Omega_{k}(G))^{p}|=|G/\Omega_{k}(G):\Omega_{1}(G/\Omega_{k}(G))|.$$ By Lemma \ref{lemma If G is in O_k then we have a property about quotients by omega} we have $\Omega_{1}(G/\Omega_{k}(G))=\Omega_{k+1}(G)/\Omega_{k}(G)$. Thus we conclude that $$ |(G/\Omega_{k}(G))^{p}|=|G/\Omega_{k}(G):\Omega_{k+1}(G)/\Omega_{k}(G)|=|G:\Omega_{k+1}(G)| $$ and hence $|G^{p^{k+1}}|=|G:\Omega_{k+1}(G)|$ by (\ref{eqn G^p^k+1 = G^p:...}) and (\ref{eqn G^p: omega...}).
\end{proof}
Thus we have proved Theorem \ref{theorem intro regular pow struc and 3 points} (\ref{theorem letter regular power structure last condition on order of G}). Hence we have now proved all three parts of Theorem \ref{theorem intro regular pow struc and 3 points} and thus have shown that quasi-powerful $p$-groups have a regular power structure for odd primes $p$. We now know that the families of $p$-groups (for odd primes $p$) with a regular power structure include regular $p$-groups, powerful $p$-groups, potent $p$-groups and quasi-powerful $p$-groups. 

\section{An Application}
\numberwithin{theorem}{section}

\label{subsection an application}
We now seek to further generalise the results from \cite{Fernandez-Alcober2007} on the orders of commutators in powerful $p$-groups.  We will see how  Theorem \ref{ theorem letter more detailed commutator order bound}  follows as a consequence of Theorem \ref{theorem intro regular pow struc and 3 points} (\ref{thm intro: product of two elements of order at most p^i has order at most p^i}).

\begin{proof}[Proof of Theorem \ref{ theorem letter more detailed commutator order bound}]
The proof is by induction on the order of $G$. The claim is clearly true for the trivial group. Now suppose that $|G|\geq p$ and that the claim holds for all quasi-powerful groups of smaller order. Let $x,y \in G$ with $o(x) \leq p^{i+1}$ and $o(y) \leq p^{i}$ Consider the images of $x$ and $y$ in $\frac{G}{\Omega_{1}(G)}$, which we shall denote  $\bar{x}$ and $\bar{y}$. Then $o(\bar{x}) \leq p^{i}$ and $o(\bar{y}) \leq p^{i-1}$. Then by the induction hypothesis we have that $o([\bar{x}^{p^{j}}, \bar{y}^{p^{k}}]) \leq p^{i-1-j-k}$. Then lifting up we see that $[x^{p^{j}}, y^{p^{k}}]^{p^{i-j-k-1}} \in \Omega_{1}(G)$. We know that $\Omega_{1}(G)$ has exponent $p$ by Theorem \ref{theorem intro regular pow struc and 3 points}(\ref{thm intro: product of two elements of order at most p^i has order at most p^i}). Thus $[x^{p^{j}},y^{p^{k}}]$ has order at most $p^{i-j-k}.$

\end{proof}

Thus we have established Theorem \ref{ theorem letter more detailed commutator order bound}. 

\begin{remark}
In fact, the argument in the proof of Theorem \ref{ theorem letter more detailed commutator order bound} will work for any family of $p$-groups where the property that $\exp \Omega_{1}(G) \leq p $ is maintained when taking quotients. In particular the argument works unchanged for potent $p$-groups for odd primes $p$.  Hence we can obtain Theorem \ref{theorem intro detailed comm bound for potent groups}.
\end{remark}
\section{Minimal generation of subgroups}
\numberwithin{theorem}{section}

\label{section generators for subgroups}
In this section we consider an $r$-generator quasi-powerful $p$-group $G$ and show that the number of generators of any subgroup of $G$ can be bounded by a quadratic function in terms of $r$ only. We begin by recalling one of the most abelian-like properties of powerful $p$-groups; that for a powerful $p$-group the minimal number of generators of a subgroup cannot exceed that of the group (Theorem \ref{thm pow p group subgroup rank bounded}).  It is natural to ask if this can be extended to all quasi-powerful $p$-groups. 

We now recall a well known family of groups of nilpotency class $2$, as an example to demonstrate that for  $r$-generator quasi-powerful $p$-groups the minimal number of generators of a subgroup can grow quadratically in $r$. 

\begin{example}
\label{Higmans example}
In \cite[Theorem 2.1]{higmanenumeratingpgroups1960}, G. Higman shows that for any given positive integer $r$ and prime $p$, there is an $r$-generator $p$-group $H$ of nilpotency class $2$ (and thus quasi-powerful) such that $\Phi(H)=Z(H)=H^{\prime}$. Furthermore $\Phi(H)$ is elementary abelian of order $p^{\frac{1}{2}r(r+1)}$. In particular $d(\Phi(H))=\frac{1}{2}r(r+1)$.   These groups are the $p$-covering groups of the elementary abelian $p$-groups.
\end{example}

We will establish an upper bound on the minimal number of generators for subgroups of quasi-powerful $p$-groups. We are able to reduce to the case where the group is of nilpotency class at most $2$. We will need to bound the number of generators of the derived subgroup of $G$. To do this we will use the following Theorem due to Witt (see \cite[Chapter 11, Theorem 11.2.2]{hall1999theory}).
\begin{theorem}[Witt]
\label{theorem of witt}
The number of commutators of weight $n$ in $r$ generators is given by $$ \frac{1}{n} \sum_{d \mid n} \mu(d)r^{n/d},$$
where $\mu(m)$ is the M\"{o}bius function.
\end{theorem}

We can now prove the main result of this section.

\begin{proof}[Proof of Theorem \ref{theorem intro bound on rank subgroups}]
We begin by identifying the smallest subgroup $K$ of $Z(G)$ such that $G/K$ is powerful. We know that $G^{\prime}\leq G^{p}Z(G)$. Let $K$ be the smallest subgroup of $Z(G)$ such that $G^{\prime} \leq G^{p}K$.  By the minimality of $K$, we have that $G^{p} \cap K \leq \Phi(K)$. That is, any element of $G^{p} \cap K$ would be redundant as a generator of $K$.  

Consider the quotient group $\bar{G}=G/K$. The group $\bar{G}$ is a powerful $p$-group, and so $d(\bar{H}) \leq d(\bar{G})=r$. We also know that $K$ is abelian and so the number of generators of any subgroup of $K$ cannot exceed $d(K)$. Thus we obtain 
\begin{align}
    \label{eqn d(h) leq r + d(Z(G) cap phi(G)}
    d(H) \leq r + d(K).
\end{align}

We now seek to bound the minimal number of generators of $K$. We observed previously that any element of $G^{p} \cap K$ would be redundant as a generator of $K$, and thus we may assume that $G^{p}=1$. We have that $[G,G]\leq G^{p}Z(G) = Z(G)$ and so we can assume that $G$ is of nilpotency class at most $2$. 

Consider the series 

$$ K \geq K \cap [G,G] \geq 1.$$

We see that
$$\frac{K}{K \cap [G,G]}= \frac{K \cap G}{(K \cap G) \cap [G,G]} \cong \frac{(K \cap G)[G,G]}{[G,G]} \leq \frac{G}{[G,G]}.$$
It is clear that $G/[G,G]$ is an $r$-generator abelian group, and so $d\left(\frac{K}{K\cap[G,G]}\right) \leq r$. Similarly $K \cap [G,G] \leq [G,G]$. The group $[G,G]$ is central, and hence abelian. Furthermore it is generated by the commutators of weight $2$ in the generators of $G$; by Theorem \ref{theorem of witt} there are $\frac{1}{2}r(r-1)$ of these. Hence $d(K \cap [G,G]) \leq \frac{1}{2}r(r-1)$. It follows that $$d(K) \leq r+ \frac{1}{2}r(r-1) = \frac{1}{2}r(r+1).$$  
Then using (\ref{eqn d(h) leq r + d(Z(G) cap phi(G)}) we see that:
$$ d(H) \leq r +\frac{1}{2}r(r+1) = \frac{1}{2}r(r+3). $$
\end{proof}
The infinite family of groups given in Example \ref{Higmans example} demonstrate that this bound is close to best possible.

We recall Theorem \ref{theorem pow p groups prod of d cyclic groups}, that for a powerful $p$-group $G$ with $d(G)=d$, we can write $G$ as a product of $d$ cyclic subgroups. As an immediate consequence of Theorem \ref{theorem intro bound on rank subgroups} we can obtain the following generalisation. 
\begin{corollary}
Let $G$ be a quasi-powerful $p$-group with $d(G)=d$. Then $G$ is a product of at most $d + \frac{1}{2}d(d+3)$ cyclic groups. 
\end{corollary}
\begin{proof}
The result follows from the facts that  $\bar{G}=G/Z(G)$ is a powerful $p$-group with $d(\bar{G}) \leq r$, and that $Z(G)$ is an abelian group of rank at most $\frac{1}{2}d(d+3)$ by Theorem \ref{theorem intro regular pow struc and 3 points}(\ref{thm intro: product of two elements of order at most p^i has order at most p^i}).
\end{proof}

\section{A Remark on the Even Prime}
\numberwithin{theorem}{section}

\label{section on the even prime}
It often happens that the prime $p=2$ behaves differently to the odd primes. For example for powerful $2$-groups L. Wilson showed in his thesis \cite{wilsonsphd} that $\exp \Omega_{i}(G) \leq 2^{i+1}$, which is best possible (an alternate proof of this bound is given in \cite{Fernandez-Alcober2007}, Corollary 2). If we took the same definition for quasi-powerful $2$-groups as for odd primes, we see that condition (\ref{eqn omega conditon}) need not hold for quasi-powerful $2$-groups. The following example gives a $2$-group $G$ such that $G/Z(G)$ is powerful. This example shows how spectacularly the regular power structure properties (\ref{eqn pth power}), (\ref{eqn omega conditon}) and (\ref{eqn index condition}) can fail for quasi-powerful $2$-groups if we take the definition to be the same as in the case of odd primes.

\begin{example}
Let $G$ be the $2$-group given by the presentation
\begin{equation*}
\begin{split}
\langle a,b,c,d,e \mid& a^2, b^8, c^2, d^4, e^2, [a,b], [a,c], [a,d], [a,e]=b^4, \\
& [b,c], [b,d]=b^2c, [b,e]=b^4,  [c,d]=b^4, [c,e], [d,e] \rangle 
\end{split}
\end{equation*}
The structure of the group can be described as $$ (\mathbb{Z}_{2} \times ((\mathbb{Z}_{8} \times \mathbb{Z}_{2}) \rtimes \mathbb{Z}_{4})) \rtimes \mathbb{Z}_{2}. $$

All three of the conditions: (\ref{eqn pth power}), (\ref{eqn omega conditon}), (\ref{eqn index condition}) fail to hold in this group: 
\begin{enumerate}[(a)]
  
      \item In this group there are $8$ distinct squares, however $|G^{2}|=16$. For instance $c$ is not the square of any element, but $c \in G^{2}$.
    \item We have that $\exp \Omega_{1}(G)=4$. Note that  $b^{2} \in \Omega_{1}(G).$ 
    \item We have  $|G:G^{2}|=16$ but $|\Omega_{1}(G)|=64$. 
\end{enumerate}

These computations can be readily checked. This particular group can be constructed in GAP \cite{GAP4.8.4} as \texttt{SmallGroup(256, 13326)}.

\end{example}

Hence we see that the `nice' results which hold for odd primes do not hold for $2$-groups $G$ such that $G/Z(G)$ is powerful. It would be interesting to see how much of the theory could be salvaged with a modified definition. 

\section{Further Remarks}
\makeatletter
\def\@secnumfont{\bfseries}
\makeatletter
\label{section further remarks}
\numberwithin{theorem}{section}

\subsection{}
\label{end notes further references}
We discuss here a few interesting applications of powerful $p$-groups, beyond what was mentioned in the introduction. 
In the introduction we alluded to the fact that powerful $p$-groups are in some sense close to abelian groups. For an excellent example of a case when the abelian condition can be dropped, and instead the theory of powerful $p$-groups can be deployed, see the paper \cite{Fernandez-Alcober2019}.
For applications of powerful $p$-groups to the study of automorphisms of finite $p$-groups, see the excellent book \cite{khukhro_1998}.
For an application of powerful $p$-groups to the study of Engel groups, see \cite{traustason2002engel}.
Powerful $p$-groups have applications beyond finite $p$-groups, for instance they also have found uses in the study of $p$-adic analytic groups (see \cite{LUBOTZKY1987506padic} and \cite{ddms1999}).

\subsection{}
In this paper we find a new family of groups which has a regular power structure. We now know of several families of groups with regular power structure, and so it is natural to ask, can all groups with a regular power structure be classified into a finite number of families, and how does this classification depend on the prime $p$?

\subsection{}
\label{endnote to show quasi-powerful pgroups need not have Gp pow embedded}
The following example is of a quasi-powerful $3$-group $G$ such that $G^{3}$  is not powerfully embedded in $G$. 
\begin{example}
Let $G$ be the quasi-powerful $3$-group given by the following presentation $$\langle a,b,c \mid a^9, b^9, c^9, [a,c]=b, [b,c]=b^{6}, [b,c]=1 \rangle. $$ This group can be described as a semidirect product of the form $(C_{9} \times C_{9}) \rtimes C_{9}$. Notice in this group that $G^{9}=1$, however $[a^3,c]=b^3$ and so $[G^{3},G] \neq 1$ and so $G^{3}$ is not powerfully embedded in $G$. This group can be constructed in GAP as \texttt{SmallGroup(729,30)}. 
\end{example}

Thus we see that $G^{3}$ must be powerful by Theorem \ref{theorem G^p is powerful} but need not be powerfully embedded.
\subsection{}
\label{further remark on powerfully nilpotent groups}
\label{subsection powerfully nilpotent groups}
We briefly recall some notions from the theory of powerfully nilpotent groups, introduced in \cite{TRAUSTASON201980}. The theory of powerfully nilpotent groups is not fundamental to this paper, we will just observe in Remark \ref{remark H is powerfull nilpotent} that a certain characteristic subgroup is powerfully nilpotent. \par

In \cite{TRAUSTASON201980} the notions of a \textit{strongly powerful} and a \textit{powerfully nilpotent} group are introduced. Loosely speaking, a powerfully nilpotent group is a powerful $p$-group admitting a special kind of central series.\par

The reason why we may want to know that a group is powerfully nilpotent, is because the structure of these groups is very rich. For example associated to every powerfully nilpotent group is a quantity known as the \textit{powerful coclass}, and it turns out that the rank and exponent of a powerfully nilpotent group can both be bounded by this quantity (for details see \cite{TRAUSTASON201980}).

\begin{definition}
We say that a finite $p$-group $G$ is \textit{strongly powerful} if $[G,G]\leq G^{p^{2}}$.
\end{definition}

In \cite[pp. 81--82]{TRAUSTASON201980} we show that if a group is strongly powerful then it follows that it is powerfully nilpotent. 

\subsection{}
\label{further properties of H}
In Section \ref{subsection basic properties of quasi-powerful $p$-groups} we defined the subgroup $H=G^{p}Z(G)$. In fact the subgroup $H$ is strongly powerful and thus is powerfully nilpotent. As $H$ is powerfully embedded in $G$, it follows that $[H,G,G] \leq H^{p^{2}}$. Hence when we show that $H$ is strongly powerful, we may assume that $[H,G,G]=1$.

\begin{proposition}
$H$ is strongly powerful.
\label{prop H is strongly powerful}
\end{proposition}
\begin{proof}
Let $h_{1}, h_{2} \in H$. By Remark \ref{formofelementsinH} we can write $h_{1}=g_{1}^{p}z_{1}$ and $h_{2}=g_{2}^{p}z_{2}$.
Then 
\begin{align*} 
[h_{1},h_{2}] &=  [g_{1}^{p} z_{1}, g_{2}^{p} z_{2}] \\ 
 &=  [g_{1}^{p}, g_{2}^{p}] \\
 &=  [g_{1}^{p}, \underbrace{g_{2} \dots g_{2}}_{p}]
\end{align*}
Then expanding and using that $g_{1}^{p} \in H$ and  $[H,G,G]=1$ yields
\begin{align*}
    [h_{1}, h_{2}] &= [g_{1}^{p},g_{2}]^{p}. 
\end{align*}
Now since $H$ is powerfully embedded in $G$ we have that $[g_{1}^{p},g_{2}]\in H^{p}$. Thus $[g_{1}^{p},g_{2}]^{p} \in (H^{p})^{p}$ and the result follows.\end{proof}

\label{remark H is powerfull nilpotent}

\bibliographystyle{amsplain} 
\bibliography{bibliography}

\end{document}